\documentclass[12pt]{amsart}
\usepackage{amsfonts}
\usepackage{amsfonts,latexsym,rawfonts,amsmath,amssymb,amsthm}
\usepackage[plainpages=false]{hyperref}
\usepackage{cite}

\usepackage{graphicx}

\RequirePackage{color}

\numberwithin{equation}{section}

\newcommand{\beq}{\begin{equation}}
\newcommand{\eeq}{\end{equation}}
\newcommand{\beqs}{\begin{eqnarray*}}
\newcommand{\eeqs}{\end{eqnarray*}}
\newcommand{\beqn}{\begin{eqnarray}}
\newcommand{\eeqn}{\end{eqnarray}}
\newcommand{\beqa}{\begin{array}}
\newcommand{\eeqa}{\end{array}}

\newtheorem{definition}{Definition}
\newtheorem{lemma}{Lemma}
\newtheorem{remark}{Remark}

\newtheorem{theorem}{Theorem}
\newtheorem{example}{Example}
\newtheorem{corollary}{Corollary}
\title  {Mixing, Li-Yorke chaos, distributional chaos and  Kato's chaos to multiple mappings }

\begin{document}

%\address{Hongbo Zeng: Department of Mathematical Sciences, Tsinghua University,  China}

\email{zenghongbo@csust.edu.cn }

%\date{}

\bibliographystyle{plain}

%\tableofcontents
\maketitle

\baselineskip=15.8pt
\parskip=3pt

\centerline {\bf   }
\centerline {}
\vskip 0.5cm

\centerline {\bf   Hongbo Zeng}
\centerline {School of Mathematics and Statistics, Changsha University of Science and Technology}
\centerline {Changsha 410114, Hunan, China}

\vskip20pt

\noindent {\bf Abstract}:
%\begin{abstract}
Let $(X,d)$ be a compact metric space and $F=\{f_1,f_2,...,f_m\}$ be an $m$-tuple of continuous maps from $X$ to itself. In this paper, we introduce the definitions of transitivity, weakly mixing and mixing of multiple mappings $(X,F)$ from a set-valued perspective, which is the semigroup generated by $F$ based on iterated function system.  Firstly, we prove that for multiple mappings, mixing implies distributional chaos in a sequence, Li-Yorke chaos and Kato's chaos. Besides, we demonstrate that $F$ is Kato's chaos if and only if $F^k$ is Kato's chaos for any $k \in \mathbb{N}$. Finally, we construct a symbolic dynamical system to show that distributional chaos may be generated by only two strongly non-wandering points.
%\end{abstract}

%\vskip20pt

\baselineskip=15.8pt
\parskip=3pt

%\newpage

%\tableofcontents
\maketitle

\baselineskip=15.8pt
\parskip=3.0pt

%\begin{abstract}

%\end{abstract}

%\keywords{iteration invariant; distributional chaos; Li-York chaos}

%\begin{multicols}{2}
\section{Introduction}
\noindent

Chaos is a highly significant research topic in the field of topological dynamical systems. The first concept of chaos can be traced back to 1975, proposed by Li and Yorke\cite{a2}. Since then, various definitions of chaos have been proposed by researchers in different fields, such as Kato's chaos\cite{a4}, Devaney chaos\cite{a5}, distributional chaos\cite{a6}, etc. So it is an important and significant to understand the relations among the deffetent definitions. Currently, there are numerous research results regarding this field, see, for instance, papers \cite{a7,a9,a42,a43,a44} and the references therein.

In 2016, Hou and Wang\cite{a1} defined multiple mappings derived from iterated function system (IFS). Iterated function system is a significant branch of fractal theory, reflecting the fundamental facts of the world\cite{a11,a12,a13,a25,a26,a31}. Hou and Wang focused primarily on studying the Hausdorff metric entropy of multiple mappings. Additionally, they introduced the notions of Hausdorff metric Li-Yorke chaos and distributional chaos from a set-valued view. Then in \cite{a14}, authors provided a sufficient condition for $F$ to exhibit distributional chaos in a sequence and chaos in the sense of Li-Yorke. They also showed that if $F$ is Hausdorff metric disdributionally chaotic, then there exist at least two strongly nonwandering points of $F$. In \cite{a15}, authors have proved that two topological conjugacy dynamical systems to multiple mappings have simultaneously Hausdorff  metric Li-Yorke chaos (distributional chaos,respectively), Hausdorff metric Li-Yorke $\delta$-chaos is equivalent to Hausdorff metric distributional $\delta$-chaos in a sequence and that for Hausdorff metric Li-Yorke chaos, there is non-mutual implication between the multiple mappings $F=\{f_{1}, f_{2}, \cdots,
f_{m}\}$ and each element $f_i$ $(i=1, 2, \cdots,m)$ in $F$. Recently, Zhao\cite{a16} introduced the definitions of sensitivity, accessibility, and Kato's chaos of multiple mappings and gave a sufficient condition for multiple mappings to be sensitive, accessible. It is worth noting that researchers studying iterated function systems often approach the topic from a group perspective rather than a set-valued perspective. This also establishes a close connection between multiple mappings and set-valued mappings, or we can consider multiple mappings as a special case of set-valued mappings.

The above special set-valued mappings consider the image of one point as a set (a compact set), with the Hausdorff metric of set-valued mappings space, and investigate the relationships among the images of points under multiple mappings from a set-valued view. As is known to all, the main idea of the system $(X,f)$ is to point out how the point of $X$ moves. However, when it comes to the origin of life, numerical simulation, demography, economic variables, target selection, population migration, etc., simply knowing the movement of point of $X$ is not enough. Besides, it is crucial to acknowledge the valuable role of set-valued mappings in addressing complex problems involving uncertainty, ambiguity, or multiple criteria. Set-valued mappings provide versatility and flexibility, making them highly beneficial across various domains. One prominent application of set-valued mappings is in  optimization problems, where the objective is to identify the optimal set of solutions. Set-valued mappings also prove their usefulness in decision-making processes that require considering multiple criteria or preferences. Unlike assigning each data point to a single category, set-valued mappings can represent uncertainty or ambiguity by assigning data points to multiple categories. In fact, the applications of set-valued mappings are vast and diverse, encompassing numerous fields beyond those mentioned here. Therefore, studying the distance relationship of points iterated under multiple mappings from a set-valued perspective is highly significant. For more recent results about these topics, one is referred to\cite{a1,a14,a15,a22,a41} and references therein.

As we all know, it is crucial to study the properties of the chaos to multiple mappings, the relationship among the various chaos to multiple mappings, and the relationship between the chaos to multiple mappings and the chaos in the classic sense. We shall study them in this paper and forthcoming papers. In the current paper, we investigate the relationship between mixing and chaos(distributional chaos in a sequence, Li-Yorke chaos and Kato's chaos) to multiple mappings. For a single continuous self-map, Liao et al. \cite{a43} claimed that mixing implies distributively chaotic in a sequence and Li-Yorke chaos, and Wang et al. \cite{a44} claimed that weakly mixing implies Kato's chaos. Besides, we study  the iteration invariance of Kato's chaos to the multiple mappings. There are many results about iteration invariance of chaos in the classic sense(see, for example,\cite{a46,a47,a48}). Finally, we try to present an example which is Hausdorff metric distributional chaos but only has two strongly non-wandering points. This result is the supplement of the Theorem 1 in \cite{a14}.

This paper is structured as follows. In Section 2, we shall first give some definitions and lemmas about multiple mappings. In Section 3, the main conclusions will be given. In Section 4, we shall give some examples.

\section{Preparations and lemmas}

Throughout this paper, $X$ is a compact
metric space with a metric $d$, and assume that $\mathbb{N}=\{1,2,3,...\}$. Let $F=\{f_{1}, f_{2}, \cdots,
f_{m}\}$ be an $m$-tuple of continuous  self-maps on $X$. Then $F(x)=\{f_{1}(x), f_{2}(x), \ldots,
f_{m}(x)\}$ is a nonempty compact subset of $X$. Firstly, we will give some definitions about set-valued spaces (see \cite{a1}). Note
$$\mathcal{K}(X)=\{K \mid K~ \text{is a nonempty compact subset of}~ X \},$$
then $F$ is a map from $X$ to $\mathcal{K}(X)$, and the metric on $\mathcal{K}(X)$ is denoted by $d_{H}$,
which is called Hausdorff metric and defined by $$d_{H}(A, B)=\max\{dist(A, B), dist(B, A)\},$$
where $dist(A, B)=\displaystyle\sup_{x\in A}\inf_{y\in B} d(x, y)$. It is obvious that $(\mathcal{K}(X), d_{H})$ is a compact metric space. Given $m\geq 1$, for any $n\in\mathbb{N}$, $F^n : X \rightarrow \mathcal{K}(X)$ is defined by $\forall x\in X$, $$F^{n}(x)=\{f_{i_{1}}\circ f_{i_{2}}\circ\cdots
\circ f_{i_{n}}(x)\mid i_{1},i_{2}, \cdots, i_{n}=1, 2, \cdots,
m\}.$$ In fact, for any subset $A\subset X$, one can define $$F^{n}(A)=\bigcup\limits_{a\in A}F^{n}(a), F^{-n}(A)=\{x\in X\mid F^{n}(x)\subset A\}.$$ Specifically, if $A\in \mathcal{K}(X)$, then $F^{n}(A)\in \mathcal{K}(X)$. Then it is clear that $F$ naturally induces a map $\widetilde{F}^{n}: \mathcal{K}(X)\rightarrow \mathcal{K}(X)$ denoted by $$\widetilde{F}^{n}(A)=F^{n}(A), \forall A\in\mathcal {K}(X).$$ One can get that $F^{n}$ and $\widetilde{F}^{n}$ are both continuous from \cite{a1}. Notice that if each element $f_i(i=1,2,\cdots,m)$ in $F=\{f_{1}, f_{2}, \cdots,
f_{m}\}$ is the same, then the Hausdorff metric dynamical systems of multiple mappings is actually the classical dynamical systems, the former is an extension of the latter.
Denote $B(A,\varepsilon)=\{x\in X \mid dist(x,A)<\varepsilon\}$, where $A\subseteq X$.

Next some definitions of Hausdorff metric chaos from a set-valued view were introduced.
\begin{definition}
The multiple mappings $F$  is called Hausdorff metric Li-Yorke chaos if there is an uncountable set $S\subset X$ satisfying for any points $x, y\in S$ with $x\neq y$, $$\liminf_{n\rightarrow\infty} d_{H}(F^{n}(x), F^{n}(y))=0,~~~ \limsup_{n\rightarrow\infty} d_{H}(F^{n}(x), F^{n}(y))>0.$$
\end{definition}

\begin{definition}
The multiple mappings $F$  is called Hausdorff metric Li-Yorke $\delta$-chaos or Hausdorff metric uniformly Li-Yorke chaos if  there are $\delta>0$ and uncountable set $S\subset X$ satisfying for any points $x, y\in S$ with $x\neq y$, $$\liminf_{n\rightarrow\infty} d_{H}(F^{n}(x), F^{n}(y))=0,~~~ \limsup_{n\rightarrow\infty} d_{H}(F^{n}(x), F^{n}(y))>\delta.$$
\end{definition}

\begin{definition}
For any pair $(x, y)\in X\times X$ and any $n\in \mathbb{N}$, denote distributional function
$\phi_{xy}^{n}(F, \cdot): R \rightarrow [0,1]$
by $$\phi_{xy}^{n}(F, t)=\frac{1}{n}\sharp \{0\leq i\leq n-1\mid
d_{H}(F^{i}(x), F^{i}(y))<t\},$$
where $\sharp A$ denotes the cardinality of the set $A$. when $t<0$, $\phi_{xy}^{n}(F, t)=0$.

Let
$$\phi_{xy}(F, t)=\liminf_{n\rightarrow\infty}\phi_{xy}^{n}(F, t), ~~\phi_{xy}^*(F, t)=\limsup_{n\rightarrow\infty}\phi_{xy}^{n}(F, t).$$
Then $\phi_{xy}(F, t)$ and $\phi_{xy}^{*}(F, t)$ respectively denote lower and upper distributional function. The multiple mappings $F$  is said to be Hausdorff metric distributional chaos if there is an uncountable set $S\subset X$ satisfying for any $x, y\in S$ with $x\neq y$, $\phi^*_{xy}(F, t)=1$ for all $t>0$ and $\phi_{xy}(F, \varepsilon)=0$ for some $\varepsilon>0$.
\end{definition}

\begin{definition}
Suppose that $\{p_{k}\}_{k=0}^\infty$ is an increasing sequence of positive integers. For any pair $(x, y)\in X\times X$ and any $n\in \mathbb{N}$, put distributional function
$\phi_{xy}^{n}(F, \cdot, p_{k}): R \rightarrow [0,1]$
by $$\phi_{xy}^{n}(F, t, p_{k})=\frac{1}{n}\sharp \{0\leq i\leq n-1\mid
d_{H}(F^{p_{i}}(x), F^{p_{i}}(y))<t\},$$
when $t<0$, $\phi_{xy}^{n}(F, t)=0$.

Let
$$\phi_{xy}(F, t, p_{i})=\liminf_{n\rightarrow\infty}\phi_{xy}^{n}(F, t, p_{i}), ~~\phi_{xy}^*(F, t, p_{i} )=\limsup_{n\rightarrow\infty}\phi_{xy}^{n}(F, t, p_{i}).$$
Then $\phi_{xy}(F, t, p_{i})$ and $\phi_{xy}^{*}(F, t, p_{i})$ are respectively called lower distributional function and upper distributional function.
The multiple mappings $F$  is said to be Hausdorff metric distributional chaos in a sequence if there is an uncountable set $S\subset X$ satisfying for any distinct points $x, y\in S$,$\phi^*_{xy}(F, t, p_{i})=1$ for all $t>0$ and $\phi_{xy}(F, \varepsilon, p_{i})=0$ for some $\varepsilon>0$. The dynamical system $(X, F)$ is said to be Hausdorff metric distributional $\delta$-chaos in a sequence or Hausdorff metric uniformly distributively chaotic in a sequence, if there exist $\delta>0$ and uncountable set $S\subset X$ such that for any $x, y\in S$ with $x\ne y $, satisfying $\phi^*_{xy}(F, t, p_{i})=1$ for all $t>0$ and $\phi_{xy}(F, \delta, p_{i})=0$ for the $\delta$ above.
\end{definition}

\begin{remark}
It is obvious that $(X, F)$ is Hausdorff metric distributional $\delta$-chaos in a sequence (Hausdorff metric $\delta$-Li-Yorke chaos, respectively) implies that $(X, F)$ is Hausdorff metric distributional chaos in a sequence (Hausdorff metric Li-Yorke chaos, respectively) by definition.
\end{remark}

\begin{definition}\cite{a16}
The multiple mappings $F$  is said to be (Hausdorff metric) sensitive, if there is $\delta>0$ such that for any nonempty open set $U\subset X$, there exist $x,y\in U$ and $n\in \mathbb{N}$ such that $d_H(F^n(x),F^n(x))>\delta$. The multiple mappings $F$  is said to be (Hausdorff metric) accessible, if for any $\varepsilon>0$ and any nonempty open set $U,V\subset X$, there exist $x\in U,y\in V$ and $n\in \mathbb{N}$ such that $d_H(F^n(x),F^n(x))<\varepsilon$. The multiple mappings $F$  is said to be (Hausdorff metric) Kato's chaotic, if it is (Hausdorff metric) sensitive and accessible.

\end{definition}

\begin{definition}
$x\in X$ is said to be a strongly nonwandering point of $F$ if for any open neighborhood $V$ of $x$, there is $y\in X$ satisfying
$$\limsup_{n\rightarrow\infty}\frac{1}{n}\sharp\{i \mid F^i(y)\cap V\neq \emptyset,  0\leq i\leq n-1\}>0.$$
\end{definition}

  Transitivity and mixing are important research topics of dynamical systems, so we will define them from a set-valued  perspective in the following. Put $R(F^n)=\{F^n(x)\mid x\in X\}$ for all $n\in \mathbb{N}$, that is the range of $F^N$, which is a subset of $\mathcal{K}(X)$. However, since the multiple mappings is a special set-valued mappings from $X$ to $\mathcal{K}(X)$, there is no natural way to extend the notion of transitivity and mixing. Therefore, in this paper, when talking about the transitivity, weakly mixing and mixing, we consider a special multiple mappings, which satisfying  $R(F^{N+1})=R(F^N)$ for some $N\in\mathbb{N}$ and $X=F(X)$. Please remember that the map always denotes the above special multiple mappings in the rest of the paper without special illustration. For the convenience of the following, we denote the range of $F^N$ above by $R_a(F)$ and the extension of $V\subseteq X$ to $\mathcal{K}(X)$ by $e(V)=\{K\in \mathcal{K}(X) \mid K\subseteq V\}$. One can get that if $V$ is an nonempty open set of $X$, then $e(V)$ is an nonempty open set of $\mathcal{K}(X)$ from \cite{a1}.

\begin{definition}
 A map $F$ is (topological) transitive if for any nonempty open sets $U$ and $V$ of $X$ with $V \cap R_a(F)\neq \emptyset$, there exists a $n\in \mathbb{N}$ such that $F^n(U)\cap e(V) \neq \emptyset$, where $F^n(U)\cap e(V) \neq \emptyset$ means that $\exists x\in U$ such that $F^n(x)\in e(V)$.
\end{definition}

\begin{definition}
 A map $F$ is (topological) weakly mixing if $F\times F: X\times X\rightarrow \mathcal{K}(X)\times \mathcal{K}(X)$ is transitive, that is  for any nonempty open sets $U_1,U_2$ and $V_1,V_2$ of $X$ with $V_1\cap R_a(F)\neq \emptyset$ and $V_2\cap R_a(F)\neq \emptyset$, there exists a $n\in \mathbb{N}$ such that $F^n(U_i)\cap e(V_i) \neq \emptyset,i=1,2$.
\end{definition}

\begin{definition}
 A map $F$ is (topological) mixing if for any nonempty open sets $U$ and $V$ of $X$ with $V\cap R_a(F)\neq \emptyset$, there exists a $N\in \mathbb{N}$ such that for all $n\geq N$, $F^n(U)\cap e(V) \neq \emptyset$.
\end{definition}

\begin{remark}
By definitions, it can be easily verified that mixing implies weakly mixing, weakly mixing imply transitive, and $F$ is (topological) mixing implies $F_m$ is transitive for any $m\in \mathbb{N}$, where $F_m$ is the Cartesian product of $m$ times $F$.
\end{remark}

For the purposes of the following, we recall some definitions of one-sided symbolic dynamical systems.
\begin{definition}
The one-sided sequence space $\Sigma_2:=\{\alpha=a_0a_1...\mid a_i=1 \text{or} 2,i\geq0\}$ is a metric space with the distance $\rho(\alpha,\beta)=\Sigma_{i=0}^\infty d(a_i,b_i)/2^i$, where $\alpha=a_0a_1,...,\beta=b_0b_1,...\in \Sigma_2,d(a_i,b_i)=1$ if $a_i\neq b_i$ and $d(a_i,b_i)=0$ if $a_i= b_i$ for $i\geq0$. The shift map $\sigma:\Sigma_2\rightarrow \Sigma_2$ is $\sigma(a_0a_1...)=(a_1a_2...)$. It is well know that $\sigma$ is a continuous map and $(\Sigma_2,\sigma)$ is a compact symbolic dynamical system. For $\alpha=(a_0a_1...)\in\Sigma_2$, let $\alpha[i,j]:=a_ia_{i+1}...a_j$ be the finite sequence from the $(i+1)$-th symbol to $(j+1)$-th symbol of $\alpha$.
\end{definition}

\begin{lemma}\label{yinli22}
The following are equivalent.

(i) $F$ is transitive.

(ii) For any nonempty open sets $U$ and $V$ of $X$ with $V \cap R_a(F)\neq \emptyset$, there exists a $n\in \mathbb{N}$ such that $F^{-n}(e(V))\cap U \neq \emptyset$.

(iii) For any nonempty open sets $U$ and $V$ of $X$ with $V \cap R_a(F)\neq \emptyset$ and any $k>0$, there exists a positive integer $n>k$ such that $F^n(U)\cap e(V) \neq \emptyset$.
\end{lemma}
\begin{proof}
$(iii)\Rightarrow (i)$ is obvious.

$(i)\Rightarrow (ii)$. Since $F$ is transitive, there exist $x_0 \in U$ and $n \in \mathbb{N}$, such that $F^n(x_0)\in e(V)$. Put $A=F^n(x_0)$. Then we have $A\in e(V)$ such that $x_0 \in F^{-n}(e(V))$. So $F^{-n}(e(V))\cap U \neq \emptyset$.

$(ii)\Rightarrow (iii)$. For any nonempty open sets $U$ and $V$ of $X$ with $V \cap R_a(F)\neq \emptyset$, we have that there exists a $n_0\in \mathbb{N}$ such that $F^{-n_0}(e(V))\cap U \neq \emptyset$ by $(ii)$. Since $F^{n}$ is continuous for any $n$, $F^{-n_0}(e(V))$ is nonempty open in $X$, put ${V}_0=F^{-n_0}(e(V))$, by the property of $R_a(F)$, we have $V_0 \cap R_a(F)\neq \emptyset$. Again by using $(ii)$ repeatedly, we know that there exists a $n_1\in \mathbb{N}$ such that $F^{-n_1}(e({V}_0))\cap U \neq \emptyset$, then we easily see that $F^{n_1+n_0}(U)\cap e(V) \neq \emptyset$. Using the same way $k$ times, we have know that there exist positive integer $n_2,n_3,...,n_k$ such that $F^{n_0+n_1+...+n_k}(U)\cap e(V) \neq \emptyset$. Let $n=n_0+n_1+...+n_k$, so we easily have that $n>k$ and $F^n(U)\cap e(V) \neq \emptyset$.
This completes the proof.
\end{proof}

%\begin{remark}
%The lemma \ref{yinli3} above is extension of the corollary 2.2 in \cite{1a}.
%\end{remark}

\begin{lemma}(\cite{a49},lemma 2.2)\label{yinli2}
There is an uncountable subset $E$ in $\Sigma_2$ such that for any different points $s=s_0s_1...,t=t_0t_1...\in E$, we have $ s_n=t_n$ for infinitely many $n$ and $s_m\neq t_m$ for infinitely many $m$, where $\Sigma_2$ denotes the symbolic dynamical system.
\end{lemma}

\begin{lemma}\label{yinli3}
Supposed $A_1,A_2\in R_a(F)$ with $A_1\neq A_2$ and $\{p_k\}_{k=1}^{\infty}$ is a sequence of positive integers. If $\forall C=C_1C_2...$, where $C_k=\overline{B(A_1,\frac{1}{k})}$ or $\overline{B(A_2,\frac{1}{k})}$, there exists $x_c\in X$, such that $\forall k\geq1$, we have $F^{p_k}(x_c)\in e(C_k)$, then $F$ is uniformly distributively chaotic in a sequence.
\end{lemma}

\begin{proof}
Take the set $E$ as in lemma \ref{yinli2}. Then by the assumptions, for any $s=s_0s_1...\in E$, there exists $x_s\in X$, such that for each $k\geq1$, $n!<k\leq (n+1)!$, we have
\[F^{p_k}(x_s)\in\begin{cases}
e(\overline{B(A_1,\frac{1}{k})}), &  s_n=0, \\
e(\overline{B(A_2,\frac{1}{k})}), &  s_n=1.
\end{cases}\]
Put $D=\{x_s\mid s\in E\}$. It is easy to see that if $s\neq t$, then $x(s)\neq x(t)$. Since $E$ is uncountable set, $D$ is uncountable set.
Let $x(s),y(t)\in D$ and $x(s)\neq y(t)$, where $s=s_0s_1...,t=t_0t_1...\in E$. By Lemma \ref{yinli2}, there exist sequences of positive integers $n_i\rightarrow\infty$ and $m_i\rightarrow\infty$ satisfying $s_{n_i}=t_{n_i},s_{m_i}\neq t_{m_i}$ for infinitely many $i$. Next we will just prove that $x(s),y(t)$ are uniformly distributively chaotic in a sequence. The whole proof is divided into two steps.

Step 1. For any $\delta>0$, we take $i$ large enough such that $\frac{1}{n_i}<\frac{\delta}{2}$, and by the property of $F^{p_k}(x)$, for $n_i!<k\leq (n_i+1)!$, we have $d_H(F^{p_k}(x), F^{p_k}(y))<\delta$. Further, we have \begin{equation*}
\begin{aligned}
&\frac{1}{n_i+1}\sharp \{1\leq k\leq (n_i+1)!\mid
d_H(F^{p_k}(x), F^{p_k}(y))<\delta\}\\
\geq&\frac{(n_i+1)!-n_i!}{(n_i+1)!}\\
=&1-\frac{1}{n_i+1}\\
\rightarrow& 1 (i\rightarrow\infty).
\end{aligned}
\end{equation*}
Therefore $$\phi_{xy}^{*}(F, \delta, p_{i})=1.$$

Step 2. Put $\varepsilon=\frac{d_H(A_1,A_2)}{2}$. Take $i$ large enough such that $\frac{1}{m_i}<\frac{d_H(A_1,A_2)}{4}$, so $d_H(F^{p_k}(x), F^{p_k}(y))>\varepsilon$ for $m_i!<k\leq (m_i+1)!$. Thus
\begin{equation*}
\begin{aligned}
&\frac{1}{m_i+1}\sharp \{1\leq k\leq (m_i+1)!\mid
d_H(F^{p_k}(x), F^{p_k}(y))<\varepsilon\}\\
\leq &\frac{m_i!}{(m_i+1)!}\\
=& \frac{1}{m_i+1}\\
\rightarrow & 0 (i\rightarrow\infty).
\end{aligned}
\end{equation*}

Therefore $$\phi_{xy}(F, t, p_{i})=0.$$

The entire proof is complete.

\end{proof}

\section{Main results}

\begin{theorem}\label{dingli1}
If $F$ is mixing, then it must be uniformly distributively chaotic in a sequence.
\end{theorem}
\begin{proof}
 Let $A_1',A_2'\in R_a(F)$ with $A_1'\neq A_2'$ and $U_0\subset X$ be any nonempty open set such that
$\overline{U_0}$ is compact set. Noting that $F$ is mixing, one has that there exists $p_1>0$, such that
$$F^{p_1}(U_0)\cap e(\overline{B(A_1',\frac{1}{k})})\neq \emptyset$$ and $$F^{p_1}(U_0)\cap e(\overline{B(A_2',\frac{1}{k})})\neq \emptyset.$$
 So one can find two points $x_1',x_2'$ satisfying $F^{p_1}(x_1')\in e(\overline{B(A_1',\frac{1}{k})})$,$F^{p_1}(x_2')\in e(\overline{B(A_2',\frac{1}{k})})$.  Next we will prove the conclusion by induction. Assuming that there exist positive integers $p_1<p_2<...<p_k$ such that for any finite sequence $A_1A_2...A_k$, where $A_i\in \{B(A_1',\frac{1}{i}), B(A_2',\frac{1}{i})\}$, there exists $x_k\in U_0$ satisfying $F^{p_i}(x_k)\in e(A_i)$ for $i=1,2,...,k$. We denote the set of all such points by $S_k$. For any $x\in S_k$, since $F^k$ is continuous for any $k\in\mathbb{N}$, there exists an open nonempty neighborhood $W_x\subset U_0$ such that $F^{p_i}(W_x)\subset e(A_i)$ ($i=1,2,...,k$). By the Lemma \ref{yinli22} and the assumptions that $F$ is mixing, $F_m$ is transitive for any $m\in \mathbb{N}$, then there exists a positive integer $p_{k+1}$ with $p_{k+1}>p_{k}$, such that
$$F^{p_{k+1}}(W_x)\cap e(\overline{B(A_1',\frac{1}{k+1})})\neq \emptyset$$ and $$F^{p_{k+1}}(W_x)\cap e(\overline{B(A_2',\frac{1}{k+1})})\neq \emptyset$$
for any $x\in S_k$. So there is a $x_{k+1}\in U_0$ satisfying $F^{p_i}(x_{k+1})\in e(A_i)$ for $i=1,2,...,k,k+1$.
By induction assumption, there is a sequence of positive integers $p_k\rightarrow\infty$  such that for each finite sequence $A_1A_2...A_k$,  there exists $x_k\in U_0$ satisfying $F^{p_i}(x)\in e(A_i)$ for $i=1,2,...,k$.
Let $C=C_1C_2...$ be an infinite sequence, where $C_k=\{e(\overline{B(A_1,\frac{1}{k})}), e(\overline{B(A_2,\frac{1}{k})})\}$. For any $k$, there exists point $x_k\in \overline{U_0}$ such that $F^{p_i}(x_k)\in e(C_i)$ for $i=1,2,...,k$. Noting that $\overline{U_0}$ is compact, one gets that the sequence $\{x_i\}_{i=1}^{\infty}$ converges to a point $x_c$ in $\overline{U_0}$. It is easily verified that for any $k\in \mathbb{N}$, we have $F^{p_i}(x_c)\in C_k$. Therefore, by Lemma \ref{yinli3}, $F$ is uniformly distributively chaotic in a sequence.

\end{proof}

%Notice that uniformly distributively chaotic in a sequence implies distributively chaotic in a sequence and Li-Yorke chaos by definitions, so we have the following result by the above theorem.
\begin{corollary}\label{tuilun}
If $F$ is mixing, then it must be distributively chaotic in a sequence and Li-Yorke chaos.
\end{corollary}

\begin{theorem}\label{ag}
Mixing implies Kato's chaos.
\end{theorem}

\begin{proof}
(i) For any $\varepsilon>0$ and nonempty open sets $U,V,W$ of $X$ with $diam(W)<\varepsilon/2$.
Since $F$ is mixing, there exists $n\geq1$ such that $F^n(U)\cap W \neq \emptyset$ and $F^n(V)\cap W \neq \emptyset$.
That is $\exists x\in U,y\in V$ such that $F^n(x)\in e(W)$ and $F^n(y)\in e(W)$. Therefore $d_H(F^n(x),F^n(y))\leq diam(W)<\varepsilon$.

(ii) Take $x_1,x_2\in X$ with $x_1\neq x_2$, Note $r=d(x_1,x_2)$ and $\delta=r/2$.  For any nonempty open set $U$ of $X$. Since $F$ is mixing, there exists $n\geq1$ such that $F^n(U)\cap e(B(x_1,r/4)) \neq \emptyset$ and $F^n(V)\cap e(B(x_1,r/4)) \neq \emptyset$. That is $\exists y_1\in U$ such that $F^n(y_1)\in e(B(x_1,r/4))$ and  $\exists y_2\in U$ such that $F^n(y_1)\in e(B(x_2,r/4))$. So we have $d_H(F^n(y_1),x_1)\leq r/4$ and $d_H(F^n(y_2),x_2)\leq r/4$. Therefore,
$d_H(F^n(y_1),F^n(y_2))> r/2=\delta$.
\end{proof}

\begin{theorem}
 $F$ is Kato's chaos if and only if $F^k$ is Kato's chaos $(k\geq2)$.
\end{theorem}
\begin{proof}
Sufficiency is obvious. We will prove the necessity in two steps.

 Step 1. Since  $F$ is Kato's chaos, there exists a $\delta>0$ such that for any nonempty open set $U$ of $X$, there exist $x,y\in U,n\geq1$, such that
 \begin{equation}\label{lian}
d_H(F^n(x),F^n(y))>\delta.
\end{equation}

  Since $F$ is continuous and $X$ is  compact metric space, $F^j$ is uniformly continuous $(\forall j=1,2,...,k)$. Therefore, for given $\delta$ above, there exists $\delta_1>0$, such that when $d(u,v)<\delta_1$, we have
\begin{equation}\label{lianxu}
d_H(F^j(u),F^j(v))<\delta (j=1,2,...,k).
\end{equation}
 Firstly, we will claim that there exists $n>k$. If not, then take a nonempty open set $U_1\subseteq U$ with $diam(U_1)<\delta_1$, on the one hand, by (\ref{lianxu}) we have that $\forall x,y\in U_1,n\leq k$, $d(F^n(x),F^n(x))<\delta$, on the other hand, there exist $x,y\in U,n\geq1$, such that $d_H(F^n(x),F^n(x))>\delta$. This is a contradiction. So there exist $x,y\in U,n\geq k$, such that $d_H(F^n(x),F^n(x))>\delta$. Now, take $j$ satisfying $n=kq+j$, where $q,j$ are positive integer and $1\leq j\leq k$. Secondly,  we will claim that there exists $n_1\geq1$, such that $d_H(F^{kn_1}(x),F^{kn_1}(y))>\delta_1/2$. If not, then for any $n_1\geq1$, we have $d_H(f^{kn_1}(x),f^{kn_1}(y))\leq \delta_1/2$, by (\ref{lianxu}) we have that  $d_H(F^{kn_1+j}(x),F^{kn_1+j}(y))\leq \delta$, therefore $d(F^n(x),F^n(y))\leq \delta$, which is a contradiction to (\ref{lian}).

 Step 2. For any $\varepsilon>0$ and nonempty open sets $U,V$ of $X$. Since $F^{j'}$ is uniformly continuous $(\forall j'=1,2,...,k)$ we have that there exists $\delta_2>0$, such that when $d(u,v)<\delta_2$, we have
\begin{equation}\label{lianxu1}
d_H(F^{j'}(u),F^{j'}(v))<\varepsilon (\forall j'=1,2,...,k).
\end{equation}
Since  $F$ is Kato's chaos, for $\delta_2$ above, there exist $x\in U, y\in V, n'\geq1$ such that $d_H(F^{n'}(x),F^{n'}(y))<\delta_2$. Take $j'$ satisfying $n'+j'=kq'$, where $q',j'$ are positive integer and $1\leq j'\leq k$. By (\ref{lianxu1}), $d_H(F^{n'+j'}(x),F^{n'+j'}(y))<\varepsilon$, that is $d_H(F^{kq'}(x),F^{kq'}(y))<\varepsilon$. Therefore, for any $\varepsilon>0$ and nonempty open sets $U,V$ of $X$, there exist $x\in U, y\in V, q' \geq1$ such that $d_H(F^{kq'}(x),F^{kq'}(y))<\varepsilon$.

\end{proof}

\section{Examples}
Firstly, we provide some examples to show that the definition of mixing (weakly mixing and transitive) of multiple mappings is well.

%\begin{example}
%Supposed $f_1,f_2$ are continuous from $[0,1]$ to itself. Let $F=\{f_1,f_2\}$, where $f_1(x)=0,\forall x\in[0,1],f_2(x)=1,\forall x\in[0,1]$. Then $R(F)=\{\{0,1\}\}$, and $F(0)=F(1)=\{0,1\}$. So $F$ is transitive.
%\end{example}

\begin{example}
Supposed $f_1,f_2$ are continuous from $[0,1]$ to itself. Let $F=\{f_1,f_2\}$, where $f_1(x)\equiv0,\forall x\in[0,1]$ and
\[f_2=\begin{cases}
2x, &  x\in[0,1/2], \\
2-2x, &  x\in(1/2,1].
\end{cases}\]
Then $R_a(F)=\{\{0,f_2^n(x)\}\mid n\geq 1,x\in [0,1]\}$. Let $U$ and $V$ be any nonempty open sets  of $X$ with $V \cap R_a(F)\neq \emptyset$, As is well-know, $f_2$ is the tent map, further $f_2$ is mixing, So it is easily verified that $F$ is mixing.
\end{example}

\begin{remark}
The Example above can be generalized. That is if $F=\{f_1,f_2\}$, where $f_1(x)\equiv a (a\in X)$, and $f_2$ is mixing (transitive,respectively) map in the classical sense, then $F$ is mixing (transitive,respectively).
\end{remark}

\begin{example}
Supposed $f_1,f_2$ are continuous from $[0,1]$ to itself. Let $F=\{f_1,f_2\}$, where
\[f_1=\begin{cases}
2x, &  x\in[0,1/2], \\
1, &  x\in(1/2,1].
\end{cases}\]
and
\[f_2=\begin{cases}
1, &  x\in[0,1/2], \\
2-2x, &  x\in(1/2,1].
\end{cases}\]

Let

\[f=\begin{cases}
2x, &  x\in[0,1/2], \\
2-2x, &  x\in(1/2,1].
\end{cases}\]

Then $R_a(F)=\{\{0,1,f^n(x)\}\mid n>1,x\in [0,1]\}$.  Let $U$ and $V$ be any nonempty open sets  of $X$ with $V \cap R_a(F)\neq \emptyset$, As is well-know, $f$ is the tent map, further $f$ is mixing, So it is easily verified that $F$ is mixing.
\end{example}

\begin{example}
Supposed $f_1,f_2$ are continuous from $[0,1]$ to itself. Let $F=\{f_1,f_2\}$, where
\[f_1=\begin{cases}
2x, &  x\in[0,1/2], \\
2-2x, &  x\in(1/2,1].
\end{cases}\]
and
\[f_2=\begin{cases}
1-2x, &  x\in[0,1/2], \\
2x-1, &  x\in(1/2,1].
\end{cases}\]

Notice that for any $x\in[0,1]$ and any $n>0$, $F^n(x)=\{f_1^n(x),f_2^n(x)\}$ and $f_1^n(x)+f_2^n(x)=1$.
Then $R_a(F)=\{\{f_1^n(x),f_2^n(x)\}\mid n\geq1,x\in [0,1]\}=\{\{f_1^n(x),1-f_1^n(x)\}\mid n\geq1,x\in [0,1]\}$. Let $U$ and $V$ be any nonempty open sets  of $X$ with $V \cap R_a(F)\neq \emptyset$, As is well-know, $f_1$ is the tent map, further $f_1$ is mixing, So it is easily verified that $F$ is mixing.
\end{example}

\begin{remark}
The multiple mappings $F$ of Example 1 or Example 3 has been shown to be Hausdorff metric Kato's chaos by definition\cite{a14}, we can also show that it is Hausdorff metric Kato's chaos by Theorem \ref{ag}. The multiple mappings $F$ of Example 2 has been shown to be Hausdorff metric Li-Yorke chaos by definition\cite{a15}, we can also show that it is Hausdorff metric Li-Yorke chaos by Corollary \ref{tuilun}.
\end{remark}

Next we will give an example, which is Hausdorff metric distribution chaos but has only two strongly non-wandering points. We have the conclusion that if $F$ is Hausdorff metric disdributionally chaotic, then there exists at least two strongly nonwandering points of $F$ from \cite{a14}. This example indicates that Hausdorff metric distributional chaos may be generated by only two strongly non-wondering points.
Firstly, we construct a sequence symbol $\{A_n\}_{n=1}^\infty$. Let $A_1=10111$. For $n\geq1$, define $A_{n+1}=A_nO_nA_nI_nA_n$ inductively, where the $O_n$ and $I_n$ have the same length as $A_n$, and $O_n$ consists only of the symbol $0$'s while $I_n$ consists only of the symbol $1$'s. Denote by $|B|$ the length of finite symbol sequence $B$. Obviously, $|A_n|=5^n,\forall n\geq1$. As $n\rightarrow\infty$, then $A_n$ enlarge to a one-side infinitely sequence, denote by $u$. Let $X$ be the $\omega$-limit set of $u$ with the shift map $\sigma$, that is $X=\omega(u,\sigma)$, which is a subspace of $\Sigma_2$. Let $a$ be the infinitely sequence consists only of the symbol $0$'s and $b$ be the infinitely sequence consists only of the symbol $1$'s.

\begin{example}
 The  dynamical system $(X,F=\{\sigma,f_0\})$ has only two strongly nonwandering points $a$ and $b$, but is Hausdorff metric distributional chaos, where $\sigma$ is a shift
map and $f_0\equiv a$.
\end{example}

\begin{proof}
Firstly, we prove that $(X,F)$ has only two strongly nonwandering points.

Obviously, $a,b\in X$ and $a,b$ are strongly nonwandering points. For any infinitely sequence $x\notin\{a,b\}$. We  just claim that the $x$ is not strongly nonwandering point of $X$. Take $n$ large enough such that the finite symbol sequence $x[0,5^n-1]\notin \{O_n,I_n\}$. Denote $B_n=x[0,5^n-1]$. Take $m=10n$. Note that for the symbol $A_{m+1}=A_mO_mA_mI_mA_m$, $B_n$ cannot be in $O_m$ and $I_m$, so we have $$\frac{1}{5^{m+1}}\sharp \{0\leq i\leq 5^{m+1}-1\mid
A_{m+1}[i,i+5^n-1]=B_n\}\leq \frac{4}{5},$$ that is the frequency of $B_n$ in $A_{m+1}$ is no more than $ \frac{4}{5}$, the frequency denotes by $v$.
Similarly, for the symbol $A_{m+2}=A_{m+1}O_{m+1}A_{m+1}I_{m+1}A_{m+1}$, $B_n$ cannot be in $O_{m+1}$ and $I_{m+1}$, so we have $$\frac{1}{5^{m+2}}\sharp \{0\leq i\leq 5^{m+2}-1\mid
A_{m+1}[i,i+5^n-1]=B_n\}\leq \frac{4}{5}v\leq{(\frac{4}{5})}^2.$$ By induction, for any $k\geq m+1$, we have
$$\frac{1}{5^{k}}\sharp \{0\leq i\leq 5^k-1\mid
A_{k}[i,i+5^n-1]=B_n\}\leq{(\frac{4}{5})}^{k-m},$$ Therefore,
$$\lim_{k\rightarrow\infty}\frac{1}{5^{k}}\sharp \{0\leq i\leq 5^k-1\mid
A_{k}[i,i+5^n-1]=B_n\}=0.$$

By the construction of $u$, we know that for any $k\geq1$, $u$ is composed of $A_k,O_k$ and $I_k$, but $B_n$ can not be in $I_k$ and $O_k$.
So, for the infinite sequence $\{D_r\}_{r=1}^\infty$, if each $D_r$ is a finite symbol sequence with length $r$ and is in infinite sequence $u$, then we have
\begin{equation}\label{lizi}
\lim_{r\rightarrow\infty}\frac{1}{r}\sharp \{0\leq i\leq r-1\mid
D_{r}[i,i+5^n-1]=B_n\}=0.
\end{equation}

For any $y\in X$. $D_r=y[0,r-1]$ denotes finite symbol sequence of $y$ with the length $r$. Since $y\in \omega(u,\sigma)$, each $D_r$ must be in $u$. So (\ref{lizi}) works.
 Take the open set $[B_n]=\{z\in X\mid z[0,5^n-1]=B_n\}$ of $x$, then by (\ref{lizi}) we have
 \begin{equation}\label{lizi1}
\lim_{r\rightarrow\infty}\frac{1}{r}\sharp \{0\leq i\leq r-1\mid
\sigma^i(y)\in [B_n]\}=0.
\end{equation}

On the other hand, since $f_0\equiv a\notin [B_n]$, we have
 \begin{equation}\label{lizi2}
\lim_{r\rightarrow\infty}\frac{1}{r}\sharp \{0\leq i\leq r-1\mid
f_0^i(y)\in [B_n]\}=0.
\end{equation}

Therefore, combine (\ref{lizi1}) and (\ref{lizi2}), one has $$\lim_{r\rightarrow\infty}\frac{1}{r}\sharp \{0\leq i\leq r-1\mid
F^i(y)\in [B_n]\}=0.$$
By the arbitrariness of $y$, $x$ is not strongly nonwandering point of $(X,F)$.

Secondly, we prove that $(X,F)$ is Hausdorff metric disdributionally chaotic.

Let $H$ be the set $\{x=E_1E_2...E_k...\}$, where $E_k\in \{I_{2^k}A_{2^k},O_{2^k}A_{2^k}I_{2^k}A_{2^k}\}$, $\forall k\geq1$. Obviously, $H$ is uncountable set.
Denote $s_k(x)=E_1E_2...E_k$. By induction hypotheses, one can easily see that for any $ k\geq1$, the finite symbol sequence $s_k(x)$ happens to be the tail of $A_{2^k+1}$. So we have $x\in X$, further $H\subset X$. It is obvious that the set $H$ has a similar structure to $\Sigma_2$, by Lemma \ref{yinli2}, there exists an uncountable set $S\subset H$ such that for any different points $x=E_1E_2...E_k...,y=F_1F_2...F_k...$ in $S$, $E_{n}=F_{n}$ for infinitely many $n$ and $E_{m}\neq F_{m}$ for infinitely many $m$. Next we will claim that $\{x,y\}\subset S$ must be disdributionally chaotic, hence, $(X,F)$ Hausdorff metric is disdributionally chaotic.

For any $k\geq2$, put $p_k= max\{|s_{k-1}(x)|,|s_{k-1}(y)|\}$, $q_k=5^{2^k}$. Since both $s_{k-1}(x)$ and $s_{k-1}(y)$ are the tail of $A_{2^{k-1}+1}$, $p_k\leq5^{2^{k-1}+1}$, further, we have
 \begin{equation}\label{lizi3}
\lim_{k\rightarrow\infty}\frac{p_k}{q_k}=0.
\end{equation}

Case 1. If $k\geq2$ satisfying $E_k=F_k$. Without loss of generality assume that $E_k=F_k=I_{2^k}A_{2^k}$. By the construction of $x,y$, we have that both $x[p_k,q_k]=y[p_k,q_k]$ are the part of $A_{2^k}$. Take infinitely sequence $\{k_m\}_{m=1}^\infty$ with $E_{k_m}=F_{k_m}, \forall m\geq1$. Notice that (\ref{lizi3}), We have
$$\lim_{m\rightarrow\infty}\frac{1}{q_{k_m}}\sharp \{0\leq i\leq q_{k_m}-1\mid
x[i,i+s]=y[i,i+s]\}=1,\forall s\geq0.$$
So $$\lim_{m\rightarrow\infty}\frac{1}{q_{k_m}}\sharp \{0\leq i\leq q_{k_m}-1\mid
d(\sigma^i(x),\sigma^i(y))<t\}=1,\forall t>0.$$
Therefore
 \begin{equation}\label{lizi4}
\limsup_{n\rightarrow\infty}\frac{1}{n}\sharp \{0\leq i\leq n-1\mid
d(\sigma^i(x),\sigma^i(y))<t\}=1,\forall t>0.
\end{equation}
Notice that $f_0^i(x)=f_0^i(y)=a, \forall i>0$ and (\ref{lizi4}), then we have
$$\limsup_{n\rightarrow\infty}\frac{1}{n}\sharp \{0\leq i\leq n-1\mid
d_H(F^i(x),F^i(y))<t\}=1,\forall t>0.$$
That is $\phi^*_{xy}(F, t)=1$ for all $t>0$.

Case 2. If $k\geq2$ satisfying $E_k\neq F_k$. Then for $x[p_k,q_k]$ and $y[p_k,q_k]$, one is part of $O_{2^k}$ and the other is part of $I_{2^k}$. Take infinitely sequence $\{k_m\}_{m=1}^\infty$ with $E_{k_m}\neq F_{k_m}, \forall m\geq1$. Notice that (\ref{lizi3}), We have
$$\lim_{m\rightarrow\infty}\frac{1}{q_{k_m}}\sharp \{0\leq i\leq q_{k_m}-1\mid
d(\sigma^i(x),\sigma^i(y))=1\}=1.$$
Therefore
 \begin{equation}\label{lizi5}
\liminf_{n\rightarrow\infty}\frac{1}{n}\sharp \{0\leq i\leq n-1\mid
d(\sigma^i(x),\sigma^i(y))<1\}=0.
\end{equation}
Notice that $f_0^i(x)=f_0^i(y)=a, \forall i>0$ and (\ref{lizi5}), we have
$$\liminf_{n\rightarrow\infty}\frac{1}{n}\sharp \{0\leq i\leq n-1\mid
d_H(F^i(x),F^i(y))<1\}=0.$$
That is $\phi_{xy}(F, 1)=0$. The entire proof is complete.

\end{proof}

%\section*{Acknowledgments}

%This research was supported by the Scientific Research Foundation of Hunan Provincial Education Department(No.23C0148, No.22C0147). The author would like to express their gratitude to the anonymous
%referees for their valuable comments and suggestions.

%\bibliography{mybibfile}

\end{document}